\documentclass[12pt]{elsarticle}

\usepackage[all,cmtip]{xy}
\usepackage{amsmath, amssymb}
\usepackage{graphics}
\usepackage{caption}
\usepackage{subcaption}
\usepackage{amsthm}
\usepackage{algorithm}
\usepackage{algorithmic}
\usepackage{booktabs}
\usepackage{color}

\newtheorem{theorem}{{\bf Theorem}}

\newtheorem{corollary}[theorem]{{\bf Corollary}}
\newtheorem{proposition}[theorem]{{\bf Proposition}}
\newtheorem{lemma}[theorem]{{\bf Lemma}}

\usepackage{stmaryrd}

\newcommand\lrb[1] {\llbracket #1 \rrbracket}

\def\bfa{\boldsymbol{a}}
\def\bfb{\boldsymbol{b}}
\def\bfj{\boldsymbol{j}}

\def\bfx{\boldsymbol{x}}
\def\bfy{\boldsymbol{y}}

\def\RR{\mathbb{R}}

\def\CCC{\mathcal{C}}

\begin{document}

\begin{frontmatter}
\title{On the Problem of Detecting When Two Implicit Plane Algebraic Curves Are Similar}


\author[a]{Juan Gerardo Alc\'azar\fnref{proy,proy2}}
\ead{juange.alcazar@uah.es}
\author[b]{Gema M. Diaz-Toca\fnref{proy2} }
\ead{gemadiaz@um.es}
 \author[a]{Carlos Hermoso}
\ead{carlos.hermoso@uah.es}

\address[a]{Departamento de F\'{\i}sica y Matem\'aticas, Universidad de Alcal\'a,
E-28871 Madrid, Spain}
\address[b]{Departamento de Matem\'atica Aplicada, Universidad de Murcia,  E-30100 Murcia, Spain}

\fntext[proy]{
Member of the Research Group {\sc asynacs} (Ref. {\sc ccee2011/r34}) }
\fntext[proy2]{
Supported by the Spanish ``Ministerio de
Econom\'ia y Competitividad" under the Project MTM2014-54141-P. }


\begin{abstract}
We make use of the complex implicit representation in order to provide a deterministic algorithm for checking whether or not two implicit algebraic curves are related by a similarity, a central question in Pattern Recognition and Computer Vision. The algorithm has been implemented in the computer algebra system Maple 2015. Examples  and evidence of the good practical performance of the algorithm are given. 
\end{abstract}

\end{frontmatter}

\section{Introduction}\label{section-introduction}

\noindent A central problem in Pattern Recognition is to classify a given object. If the object to be studied is represented as a curve, the problem reduces to comparing this curve with other curves, previously classified and stored in a database. However, the position and size of the curves in the database do not necessarily coincide with the position and size of the curve to be classified. Therefore, from a geometric point of view we need to check whether or not there exists some \emph{similarity}, i.e. the composition of a rigid motion and a scaling, transforming the given curve into some curve in the database. In Computer Vision we find a close problem, called \emph{pose estimation}, which has to do with identifying a certain object in a scene. Geometrically, the problem is equivalent to recognising a same object in two different positions. In turn, this amounts to checking if two objects are the same up to the composition of a translation and a rotation, so that the scaling factor does not change. 

These questions and related problems have been attacked in the applied mathematics literature using very different approaches: B-splines \cite{Huang}, Fourier descriptors \cite{Sener}, complex representations \cite{TC00}, M\"oebius transformations in the parameter space \cite{AHM13}, statistics \cite{Gal, Lei, Mishra} (also in the 3D case), moments \cite{Suk1, Taubin2}, geometric invariants \cite{Unel, WU98a, WU98b}, Newton-Puiseux parametrizations \cite{MM02}, or differential invariants \cite{Boutin, Calabi, Weiss}. The interested reader may consult the bibliographies in these papers to find other references on the matter; the list is really very long. 

In the above references the input is typically a ``fuzzy" image, possibly noisy, with occluded parts, etc. and quite often a point cloud. A common strategy is to adjust first an algebraic curve to the input, so that the comparison is carried out between algebraic curves. Although in some cases (B-splines \cite{Huang}, Fourier descriptors \cite{Sener}) the curves are described by parametric representations, in most approaches the curves are represented by means of implicit algebraic equations. Therefore, the question of detecting similarity between implicit algebraic curves comes up.

Since in the given references the input is usually noisy and the representation of the object as an algebraic curve is only an approximation, the comparison made between the curves is also approximate. In particular, statistical and numerical methods are massively employed. However, in this paper we address the problem from a different perspective. We assume that the curves are given in \emph{exact} arithmetic, and we provide a deterministic, not approximate, answer to the question whether or not these two curves are similar. In the affirmative case, our algorithm can also find the similarities transforming one into the other. The motivation for this point of view comes from computer algebra systems. Assume that a database of classical curves is stored in your favourite computer algebra system. Using the algorithms in this paper, the system can recognise a certain curve introduced by the user as one of the curves in the database. This way, the user can identify a curve as, say, a cardioid, a lemniscata, an epitrochoid, a deltoid, etc.

In order to do this, we use the \emph{complex representations} of the curves to be compared. Complex representations have already been used in \cite{TC00, UW}, where the pose-estimation problem is addressed, and in \cite{LR08}, where the computation of the symmetries of an algebraic planar curve is studied. In \cite{TC00} the complex representation is exploited and combined with numerical strategies in order to give an approximate solution to the problem. In \cite{UW} the transformation to be sought is decomposed into a rotation, which is approached in a deterministic way, and a translation, which is approximated. The approach in \cite{LR08} is deterministic, but they treat a different problem. In our case, we use the complex representation of the curves to translate the problem into a bivariate polynomial system of equations with complex coefficients. Additionally, except in certain special cases we can also add a univariate equation to the system (corresponding to the rotation angle), which allows to solve the system in a fast and efficient way.

Besides complex representations, we also employ usual tools in the field of computer algebra, like resultants, gcds and polynomial system solving. The resulting algorithm has been implemented and tested in the computer algebra system Maple 2015. Our experiments confirm that the algorithm is efficient and fast. 

The structure of the paper is the following. Preliminaries are introduced in Section \ref{sec-prelim}. In Section \ref{sec-main} we discuss the two cases that must be distinguished. The most general case is addressed in Section \ref{sec-general}, and a special case is studied in Section \ref{sec-special}. Both in Section \ref{sec-general} and Section \ref{sec-special} we focus on orientation-preserving similarities. Although the orientation-reversing case is analogous, some observations for this case are provided in Section \ref{sec-or-reverse}. Finally, examples and experimentation details are provided in Section \ref{sec-impl}.

\section{Preliminaries} \label{sec-prelim}
\noindent Throughout the paper, we consider two plane, real, algebraic curves $\CCC_1,\CCC_2 \subset \RR^2$ implicitly defined by two real, irreducible polynomials $f(x,y), g(x,y)$ of the same degree $n$ with rational coefficients. Furthermore, we will write
\[f(x,y)=f_n(x,y)+f_{n-1}(x,y)+\cdots +f_0(x,y),\]where for $p=0,1,\ldots,n$, $f_p(x,y)$ represents the homogeneous form of degree $p$ of $f$,
\[f_p(x,y)=\sum_{j=0}^p a_{p-j,j}x^{p-j}y^j.\]Similarly,
\[g(x,y)=g_n(x,y)+g_{n-1}(x,y)+\cdots +g_0(x,y), \mbox{   }g_p(x,y)=\sum_{j=0}^p b_{p-j,j}x^{p-j}y^j.\]
We will assume that $\CCC_1,\CCC_2 \subset \RR^2$ are {\it real} curves, i.e. that they contain infinitely many real points, and that they are not either lines or circles. We have found it useful to identify the Euclidean plane with the complex plane. Through this correspondence, $(x,y) \simeq z=x + iy$ and
\begin{equation}\label{eq:complexchange}
x=\frac{z+\bar{z}}{2},\mbox{ }y=\frac{z-\bar{z}}{2i},
\end{equation}
where $\bar{z}$ denotes the conjugate of the complex number $z$. By performing the change \eqref{eq:complexchange} into the implicit equation $f(x,y)$, we reach a polynomial $F(z,\bar{z})$ that can be regarded as a ``complex" implicit equation of $\CCC_1$, since it is, as it also happens with the usual implicit equation, unique up to multiplication by (possibly complex) constants. We will refer to $F(z,\bar{z})$ as the \emph{complex implicit representation} of ${\mathcal C}_1$. The polynomial $F(z,\bar{z})$ can be written as
\[
F(z,\bar{z})=F_n(z,\bar{z})+F_{n-1}(z,\bar{z})+\cdots +F_0(z,\bar{z}),
\]
where $F_p(z,\bar{z})$, $p=0,1,\ldots,n$, is an homogeneous polynomial of degree $p$ in the variables $z,\bar{z}$, i.e. for $p\geq 1$,
\[
F_p(z,\bar{z})=\sum_{j=0}^p \alpha_{p-j,j}z^{p-j}\bar{z}^j. 
\]
Furthermore, one can check that for $p=0,1,\ldots,n$, $f_p(x,y)$ gives rise to $F_p(z,\bar{z})$ and conversely. We will say that a term in $F(z,\bar{z})$ has \emph{bidegree} $(n-k,k)$ if it can be written as $\xi\cdot z^{n-k}\bar{z}^k$, with $\xi\in {\Bbb C}$. In a similar way, by applying \eqref{eq:complexchange} into $g(x,y)$ we reach $G(z,\bar{z})$, where
\[
G(z,\bar{z})=G_n(z,\bar{z})+G_{n-1}(z,\bar{z})+\cdots +G_0(z,\bar{z}),\mbox{   }G_p(z,\bar{z})=\sum_{j=0}^p \beta_{p-j,j}z^{p-j}\bar{z}^j. 
\]

 An \emph{(affine) similarity} of the plane is a linear affine map from the plane to itself that preserves ratios of distances. More precisely, a map $h: \RR^2 \longrightarrow \RR^2$ is a similarity if $h(\bfx) = A\bfx + b$ for an invertible matrix $A\in \RR^{2\times 2}$ and a vector $b\in \RR^2$, and there exists $\textsc{r} > 0$ such that
\[ \|h(\bfx) - h(\bfy)\|_2 = \textsc{r}\cdot \|\bfx - \bfy\|_2,\qquad \bfx,\bfy\in \RR^2, \]
where $\|\cdot\|_2$ denotes the Euclidean norm. We refer to $\textsc{r}$ as the \emph{ratio} of the similarity. Notice that if $\textsc{r} = 1$ then $h$ is an \emph{(affine) isometry}, i.e. it preserves distances. Planar isometries are completely classified \cite{Coxeter}. Under the composition operation, similarities form a group, and isometries form a subgroup inside of this group. Furthermore, any similarity can be decomposed into an isometry and a uniform scaling by its ratio $\textsc{r}$. We say that $\CCC_1,\CCC_2$ are similar if and only if there exists a similarity $h$ such that $h(\CCC_1)=\CCC_2$. Notice that if $\CCC_1,\CCC_2$ are similar, since by hypothesis they do not have multiple components their degrees must coincide. This explains why we assumed the equality of the degrees of $\CCC_1,\CCC_2$ at the beginning of the paper. Furthermore, whenever $\CCC_1,\CCC_2$ are not lines or circles the number of similarities between them is finite \cite{AHM13}; in fact, if $\CCC_1,\CCC_2$ are not symmetric then there exists at most one similarity mapping one onto the other \cite{AHM13}.

A similarity can either preserve or reverse the orientation. In the first case we say that it is \emph{orientation-preserving}, and in the second case we say that it is \emph{orientation-reversing}. By identifying ${\Bbb R}^2$ and ${\Bbb C}$, any orientation-preserving similarity can be written as $h(z)=\bfa z+\bfb$, where $\bfa,\bfb\in {\Bbb C}$. Furthermore, any orientation-reversing similarity can be written as $h(z) = \bfa \overline{z} + \bfb$. In each case, its ratio $\textsc{r} = |\bfa|\neq 0$. The next result shows the relationship between the complex implicit equations of two similar curves $\CCC_1,\CCC_2$.

\begin{theorem}\label{relat}
Let $\CCC_1,\CCC_2$ be two algebraic curves with complex implicit equations $F(z,\bar{z})$ and $G(z,\bar{z})$ with the same degree. The curves $\CCC_1,\CCC_2$ are related by an orientation-preserving similarity $h(z)=\bfa z +\bfb$, if and only if there exists $\lambda\in {\Bbb C}$, $\lambda\neq 0$ such that 
\begin{equation}\label{eq:eqimplicits-1}
G(\bfa z+\bfb,\overline{\bfa}\bar{z}+\overline{\bfb})=\lambda \cdot F(z,\bar{z})
\end{equation}
for all $z\in {\Bbb C}$. Also, $\CCC_1,\CCC_2$ are related by an orientation-reversing similarity $h(z)=\bfa \bar{z} +\bfb$, if and only if there exists $\lambda $ such that 
\begin{equation}\label{eq:eqimplicits-2}
G(\bfa \bar{z}+\bfb,\overline{\bfa}z+\overline{\bfb})=\lambda \cdot F(z,\bar{z})
\end{equation}
for all $z\in {\Bbb C}$.
\end{theorem}

\begin{proof} We prove the result for the orientation-preserving case. Similarly for the orientation-reversing case. Here we identify each point $(x,y)\in {\Bbb R}^2$ with the complex number $z=x+iy$. $(\Rightarrow)$ If $\CCC_1,\CCC_2$ are related by an orientation-preserving similarity $h(z)=\bfa z +\bfb$, then for every point $z\in \CCC_1$ the point $\omega=h(z)$ belongs to $\CCC_2$, i.e. $G(\omega,\overline{\omega})=0$. Since $\omega=\bfa z +\bfb$ the curves defined by $F(z,\bar{z})=0$ and $G(\bfa z+\bfb,\overline{\bfa}\bar{z}+\overline{\bfb})=0$ have infinitely many points in common. By Bezout's Theorem and since they have the same degree, \eqref{eq:eqimplicits-1} follows. $(\Leftarrow)$ If \eqref{eq:eqimplicits-1} holds then whenever $z\in \CCC_1$, i.e. whenever $F(z,\bar{z})=0$, then $G(\bfa z+\bfb,\overline{\bfa}\bar{z}+\overline{\bfb})=0$, i.e. $h(z)$ is a point of $\CCC_2$. Therefore, the image $h(\CCC_1)$, which is an algebraic curve, must be a component of $\CCC_2$. Since $\CCC_2$ and $h(\CCC_1)$ have the same degree $\CCC_2=h(\CCC_1)$.
\end{proof}

Notice that in the orientation-preserving case, $h(z)=\bfa z + \bfb$ is the composition of: (1) a rotation around the origin, of angle equal to $\theta$, where $\bfa=|\bfa|\cdot e^{i\theta}$; (2) a scaling, where the scaling factor is $|\bfa|$ and the center is the origin; (3) a translation of vector $\bfb$. In the orientation-reversing case, $h(z) = \bfa \overline{z} + \bfb$ is the composition of: (1) a symmetry with respect to the $x$-axis; (2) a rotation around the origin, of angle equal to $\theta$, where $\bfa=|\bfa|\cdot e^{i\theta}$; (3) a scaling, where the scaling factor is $|\bfa|$ and the center is the origin; (4) a translation of vector $\bfb$.

\section{Similarity detection: discussion of cases.} \label{sec-main}

The equations \eqref{eq:eqimplicits-1} and \eqref{eq:eqimplicits-2}, jointly with the conditions $\bfa,\lambda\neq 0$ are necessary and sufficient for $\CCC_1,\CCC_2$ to be similar.
Equaling the terms in $z^{m-j}\bar{z}^j$, $0\leq j\leq m$, $0\leq m\leq n$ at both sides of \eqref{eq:eqimplicits-1}
leads to a polynomial system ${\mathcal S}$ of at most $N={n+2 \choose 2}$ equations with complex coefficients in the variables $\bfa,\overline{\bfa},\bfb,\overline{\bfb}, \lambda$. The solutions to that system with $\bfa,\lambda\neq 0$ correspond to the similarities between $\CCC_1,\CCC_2$ preserving orientation. Similarly, by considering \eqref{eq:eqimplicits-2} we reach another polynomial system ${\mathcal S}'$ whose solutions correspond to the similarities reversing orientation. Thus, $\CCC_1,\CCC_2$ are similar if and only if either ${\mathcal S}$ or ${\mathcal S}'$ is consistent and has a solution with $\bfa,\lambda\neq 0$. In the sequel, we will show how to efficiently check the consistency of ${\mathcal S}$. For ${\mathcal S}'$ the analysis is similar; we will make some observations on this case in Section \ref{sec-or-reverse}. 

Since $\CCC_1$ has degree $n$, there exists $j\in \{0,1,\ldots,n\}$ such that $\alpha_{n-j,j}\neq 0$. Let us see that we can always find $j\neq n$ with this property. We need the following lemma. This result is mentioned without a proof in \cite{TC00} and \cite{LR08}, so for completeness we include a proof here.                      

\begin{lemma} \label{lem-alphas}
For $j\in \{0,1,\ldots,n\}$ we have $\alpha_{n-j,j}=\overline{\alpha}_{j,n-j}$, $\beta_{n-j,j}=\overline{\beta}_{j,n-j}$.
\end{lemma}

\begin{proof} We prove the result for the $\alpha$'s; the proof for the $\beta$'s is analogous. Performing the change \eqref{eq:complexchange} in each term $a_{n-k,k}x^{n-k}y^k$, where $k\in \{0,1,\ldots,n\}$, we get
\[
\begin{array}{l}
a_{n-k,k}\left(\frac{z+\bar{z}}{2}\right)^{n-k}\cdot \left(\frac{z-\bar{z}}{2i}\right)^k=\\
=\frac{a_{n-k,k}}{2^n \cdot i^k}\cdot \left[\sum_{j_1=0}^{n-k}{n-k \choose j_1}z^{n-k-j_1}\cdot \bar{z}^{j_1}\right]\cdot \left[\sum_{j_2=0}^k {k \choose j_2} (-1)^{j_2} z^{k-j_2}\cdot \bar{z}^{j_2}\right].
\end{array}
\]If we expand this multiplication, whenever $j_1+j_2=j$ we obtain a term in $z^{n-j}\bar{z}^j$; therefore,
\begin{equation}\label{eq:alpha}
\alpha_{n-j,j}=\sum_{k=0}^n \frac{a_{n-k,k}}{2^n\cdot i^k}\cdot\left[ \sum_{\begin{array}{c}j_1+j_2=j\\ j_1\leq n-k,j_2\leq k \end{array}} {n-k \choose j_1}{k \choose j_2} (-1)^{j_2}\right].
\end{equation}
Similarly, in order to find $\overline{\alpha}_{j,n-j}$ we substitute \eqref{eq:complexchange} into $a_{n-k,k}x^{n-k}y^k$, but we proceed in a slightly different way:
\[
\begin{array}{l}
a_{n-k,k}\left(\frac{\bar{z}+z}{2}\right)^{n-k}\cdot \left(\frac{-\bar{z}+z}{2i}\right)^k=\\
=\frac{a_{n-k,k}}{2^n \cdot i^k}\cdot \left[\sum_{\ell_1=0}^{n-k}{n-k \choose \ell_1}\bar{z}^{n-k-\ell_1}\cdot z^{\ell_1}\right]\cdot \left[\sum_{\ell_2=0}^k {k \choose \ell_2} (-1)^{\ell_2} \bar{z}^{k-\ell_2}\cdot z^{\ell_2}\right].
\end{array}
\]When $\ell_1+\ell_2=j$ we get $z^j\bar{z}^{n-j}$; hence,
\begin{equation}\label{eq:alpha-2}
\alpha_{j,n-j}=\sum_{k=0}^n \frac{a_{n-k,k}\cdot(-1)^k}{2^n\cdot i^k}\cdot\left[ \sum_{\begin{array}{c}\ell_1+\ell_2=j\\ \ell_1\leq n-k,\ell_2\leq k \end{array}} {n-k \choose \ell_1}{k \choose \ell_2} (-1)^{-\ell_2}\right].
\end{equation}
Finally, comparing the terms in \eqref{eq:alpha} and \eqref{eq:alpha-2} we notice that when $k$ is even we get the same real number, and when $k$ is odd we get two opposite purely imaginary numbers. So the result follows.
\end{proof}

If $\alpha_{0,n}\neq 0$ then from Lemma \ref{lem-alphas} we have that $\alpha_{n,0}\neq 0$ too; therefore there always exists $j\in \{0,1,\ldots,n-1\}$ such that $\alpha_{n-j,j}\neq 0$. Now if $\alpha_{n-j,j}\neq 0$ and $\beta_{n-j,j}=0$, the equality \eqref{eq:eqimplicits-1} cannot be satisfied, and therefore ${\mathcal C}_1$, ${\mathcal C}_2$ are not related by an orientation-preserving similarity. Thus, in the sequel whenever $\alpha_{n-j,j}\neq 0$ we will assume $\beta_{n-j,j}\neq 0$ too. Since
\[
G(\bfa z + \bfb,\overline{\bfa}\bar{z}+\overline{\bfb})=\cdots+\beta_{n-j,j}(\bfa z+\bfb)^{n-j}(\overline{\bfa}\bar{z}+\overline{\bfb})^j+\cdots,
\]
by equaling the coefficients of $z^{n-j}\bar{z}^{j}$ at both sides of \eqref{eq:eqimplicits-1}, we get
\begin{equation}\label{eq:lamda}
\lambda=\frac{\beta_{n-j,j}\bfa^{n-j}\overline{\bfa}^{j}}{\alpha_{n-j,j}}.
\end{equation}

\noindent Now let us compare the coefficients of $z^{n-j-1}\bar{z}^j$ at both sides of \eqref{eq:eqimplicits-1}. In order to find the coefficient of $z^{n-j-1}\bar{z}^{\bfj}$ in $G(\bfa z + \bfb,\overline{\bfa}\bar{z}+\overline{\bfb})$, we observe that the terms of $G(z,\bar{z})$ contributing to the term $z^{n-j-1}\bar{z}^{j}$ in $G(\bfa z + \bfb,\overline{\bfa}\bar{z}+\overline{\bfb})$ are the terms with bidegrees $(n-j,j)$, $(n-j-1,j+1)$ and $(n-j-1,j)$. After substituting $z\to \bfa z+\bfb$ in these terms, the coefficient of $z^{n-j-1}\bar{z}^{j}$ at the left hand-side of \eqref{eq:eqimplicits-1} is  
\[(n-j)\bfa^{n-j-1}\overline{\bfa}^{j} \bfb+(j+1)\beta_{n-j-1,j+1}\cdot \bfa^{n-j-1}\overline{\bfa}^{j}\overline{\bfb}+\beta_{n-j-1,j}\cdot \bfa^{n-j-1}\overline{\bfa}^{j}.\]
Notice that this expression does not make sense when $j=n$, which is the reason why we needed $j\in \{0,\ldots,n-1\}$. The above coefficient has to be equal to the coefficient of $z^{n-j-1}\bar{z}^{j}$ in $\lambda\cdot F(z,\bar{z})$, namely $\lambda\cdot \alpha_{n-j-1,j}$. Therefore, after taking \eqref{eq:lamda} into account and dividing by $\bfa^{n-j-1}\overline{\bfa}^{j}$, we deduce that
\begin{equation}\label{eq:first}
(n-j)\beta_{n-j,j}\cdot \bfb+(j+1)\beta_{n-j-1,j+1}\cdot \overline{\bfb}+\beta_{n-j-1,j}= \frac{\beta_{n-j,j}\cdot\alpha_{n-j-1,j}}{\alpha_{n-j,j}}\cdot \bfa,
\end{equation}
which is linear in $\bfa,\bfb,\overline{\bfb}$. By conjugating \eqref{eq:first}, we also deduce the following relationship:
\begin{equation}\label{eq:second}
(j+1)\overline{\beta}_{n-j-1,j+1}\cdot \bfb+(n-j)\overline{\beta}_{n-j,j}\cdot\overline{\bfb}+\overline{\beta}_{n-j-1,j}= \frac{\overline{\beta}_{n-j,j}\cdot\overline{\alpha}_{n-j-1,j}}{\overline {\alpha}_{n-j,j}}\cdot \overline{\bfa}.
\end{equation}

Let $\Delta_j=(n-j)^2\cdot |\beta_{n-j,j}|^2-(j+1)^2\cdot |\beta_{n-j-1,j+1}|^2 $. The following observation is crucial for us: if $\Delta_j\neq 0$, then we can write $\bfb$ as a linear expression $\bfb=\eta(\bfa,\overline{\bfa})$ from the linear system consisting of \eqref{eq:first} and \eqref{eq:second}. Hence, we will distinguish the following cases. Here we assume that $\alpha_{n-j,j}\neq 0$ iff $\beta_{n-j,j}\neq 0$, since otherwise we know that ${\mathcal C}_1$ and ${\mathcal C}_2$ cannot be similar. 
\begin{itemize}
\item We will say that we are in the \emph{general case for ${\mathcal C}_2$}, if there exists some $j\in \{0,\ldots,n-1\}$ with $\alpha_{n-j,j}\neq0$ such that $\Delta_j\neq 0$.
\item We will say that we are in the \emph{special case for ${\mathcal C}_2$} if for all $j\in \{0,\ldots,n-1\}$ with $\alpha_{n-j,j}\neq0$, we have $\Delta_j=0$. 
\end{itemize}
The general or special cases for ${\mathcal C}_1$ can be introduced similarly. 

In order to quickly detect whether or not we are in the general or the special case, we can proceed as follows. Let $\bfj$ be the minimum of the $j\in \{0,\ldots,n-1\}$ such that $\alpha_{n-j,j},\beta_{n-j,j}\neq 0$  (observe that $\bfj \leq \lceil{\frac{n}{2}\rceil}$). Then the following result holds.

\begin{lemma} \label{lem-no-gen-case}
Assume that $\alpha_{n-j,j}\neq 0$ iff $\beta_{n-j,j}\neq 0$. The special case for ${\mathcal C}_2$ occurs if and only if $|\beta_{n-j,j}|={n \choose j}\cdot |\beta_{n-\bfj,\bfj}|$ for $j=0,\ldots,n-1$. In particular, if $\bfj\neq 0$ then the special case for ${\mathcal C}_2$ does not occur.
\end{lemma}
 
\begin{proof} $(\Leftarrow)$ Since $|\beta_{n-j,j}|={n \choose j}\cdot |\beta_{n-\bfj,\bfj}|$ for $j=0,\ldots,n-1$, we have
\[\Delta_j=\left[(n-j)^2 {n\choose j}^2-(j+1)^2 {n\choose j+1}^2\right]\cdot  |\beta_{n-\bfj,\bfj}|^2.\]
Since $\frac{{n \choose j+1}}{{n \choose j}}=\frac{n-j}{j+1}$, we can easily see that $\Delta_j=0$. Furthemore, since by definition $\beta_{n-\bfj,\bfj}\neq 0$, we have $\beta_{n-j,j}\neq 0$ for $j=0,\ldots,n-1$. Therefore, $\alpha_{n-j,j}\neq 0$ for $j=0,\ldots,n-1$.

$(\Rightarrow)$ By induction on $k$, one can prove that $|\beta_{n-(\bfj+k),\bfj+k}|={n \choose \bfj+k}\cdot |\beta_{n-\bfj,\bfj}|$ for $0\leq k\leq n-\bfj$. Thus, when $k=n-\bfj$ we deduce that $|\beta_{0,n}|\neq 0$. Therefore, $\beta_{0,n}\neq 0$ and by Lemma \ref{lem-alphas}, $\beta_{n,0}\neq 0$. Hence $\bfj=0$. Furthermore, since by definition $\beta_{n-\bfj,\bfj}\neq 0$, we also deduce that all the $\alpha_{n-j,j}$'s are nonzero.
\end{proof}

\begin{corollary}\label{allspecial}
If $\CCC_1,\CCC_2$ are similar, the special case for $\CCC_2$ occurs iff the special case for $\CCC_1$ occurs too.
\end{corollary}

\begin{proof} If $\CCC_1,\CCC_2$ are similar then from Equation \eqref{eq:eqimplicits-1}, we have that 
\begin{equation}\label{larel}
\bfa^{n-j}\cdot \overline{\bfa}^j\cdot \beta_{n-j,j}=\lambda\cdot \alpha_{n-j,j}
\end{equation}
for $j=0,\ldots,n$. Furthermore, by taking the absolute value in \eqref{eq:lamda} and fixing $j=\bfj$, we get that 
\begin{equation}\label{lam}
|\lambda|=\frac{|\beta_{n-\bfj,\bfj}|\cdot |\bfa|^n}{|\alpha_{n-\bfj,\bfj}|}.
\end{equation}
So from \eqref{larel} and \eqref{lam}, we get 
\[|\alpha_{n-j,j}|=\frac{|\alpha_{n-\bfj,\bfj}|\cdot |\beta_{n-j,j}|}{|\beta_{n-\bfj,\bfj}|}.\]
Then the statement follows from Lemma \ref{lem-no-gen-case}.
\end{proof}

If the special case happens for one of the curves but not the other, from Corollary \ref{allspecial} the curves cannot be similar. So from now on we will assume that if the special case happens for one curve, it also happens for the other curve, and we will refer to this situation as the \emph{special case}. In the next two sections, we provide algorithms to check similarity in each of these cases.

\section{The general case.} \label{sec-general}

By using the expressions \eqref{eq:lamda}, \eqref{eq:first} and \eqref{eq:second} we can write $\lambda, \bfb, \overline{\bfb}$ in terms of $\bfa$, $\overline{\bfa}$. Therefore, the system ${\mathcal S}$ gives rise to a bivariate system in $\bfa,\overline{\bfa}$ with complex coefficients, and $\CCC_1$ and $\CCC_2$ are similar iff this system has some solution with $\bfa \neq 0$. However, we will see that for a generic pair of curves $\CCC_1,\CCC_2$, we can do better. The rough idea is to move from the bivariate system in $\bfa, \overline{\bfa}$ to another bivariate system in two new variables, which contains at least one univariate equation. The fact that the system contains a univariate equation makes the analysis of the system much easier and faster. This idea does not work in certain special, bad situations, described below. 

Let $\bfa=|\bfa|\cdot e^{i\theta}=|\bfa|\cdot (\cos\theta+i\sin\theta)$, $\theta\in (-\pi,\pi]$. We need to consider separately the cases $\cos(\theta)\neq 0$ and $\cos(\theta)=0$. We focus first on the case $\cos\theta\neq 0$; the case $\cos\theta=0$ is easier and will be treated at the end of the section. Under the assumption $\cos\theta\neq 0$, we can write $\bfa=r\cdot (1+i\omega)$, where $r=|\bfa|\cdot \cos\theta$ and $\omega=\mbox{tg}(\theta)$. Now observe that if $\CCC_1$ and $\CCC_2$ are related by a similarity $h(z)=\bfa z+\bfb$, then $h$ maps the curve defined by the form of maximum degree of ${\mathcal C}_1$, $f_n(x,y)$, into the curve defined by the form of maximum degree of ${\mathcal C}_2$, $g_n(x,y)$. The former is mapped onto the latter by the mapping $\tilde{h}(z)=\bfa z$, which is also a similarity. The curves $f_n(x,y)=0$ and $g_n(x,y)=0$ are the union of $n$ (possibly complex) lines, counted with multiplicity. Let us denote the lines corresponding to the curve $f_n(x,y)=0$ by ${\mathcal L}_1,\ldots,{\mathcal L}_n$, and the lines corresponding to the curve $g_n(x,y)=0$ by ${\mathcal M}_1,\ldots,{\mathcal M}_n$. If ${\mathcal C}_1$ and ${\mathcal C}_2$ are similar, the ${\mathcal M}_i$'s are the result of rotating the ${\mathcal L}_i$'s about the origin by an angle equal to $\theta$. Let us see that in most situations we can either directly find $\omega$, or find a polynomial $P(t)$ such that $\omega$ is a real root of $P(t)$. In order to find $P(t)$, we distinguish two cases: 

\begin{itemize}
		
\item {\bf Case (1): $g_n(x,y)=x^k(ax+by)^{n-k}$}, $k\neq 0$, $k\neq n$, $b\neq 0$. If the curves ${\mathcal C}_1$ and ${\mathcal C}_2$ are similar then $f_n(x,y)$ must have two irreducible factors as well, with multiplicities $k,n-k$. If $x$ does not divide $f_n(x,y)$ we can move to the second case. Otherwise $f_n(x,y)=x^\ell(cx+dy)^{n-\ell}$, where $k=\ell$ or $k=n-\ell$, $\ell\neq 0$. Now we have several subcases:  
    \begin{itemize}
    \item If $a=c=0$ then $f_n(x,y)=x^\ell y^{n-\ell}$ and $g_n(x,y)=x^k y^{n-k}$. In this case, we proceed as follows: (1) if $k\neq n-k$, then if $k=\ell$ then $\omega=0$, i.e. it is a root of $P(t)=t$, and if $k\neq \ell$ then $\cos\theta=0$, which has been postponed; (2) if $k= n-k$, then either $\omega=0$ or $\cos\theta=0$, and both possibilities must be explored.
    \item If $a\neq 0$ or $c\neq 0$, then:
    \begin{itemize}
    \item If $k=\ell$ and $n-k\neq k$, then $\omega=0$, i.e. $\omega$ is a root of $P(t)=t$.
    \item If $k\neq \ell$ then $\omega=-d/c$, i.e. $\omega$ is a root of $P(t)=t+d/c$.
    \item If $k=n-k=\ell$ then $\omega$ is a root of $P(t)=t(t+d/c)$.
    \end{itemize}
    \end{itemize}
\item {\bf Case (2): case (1) does not occur}. Here, both $f_n$ and $g_n$ have factors corresponding to lines $y=mx$. In the case of $f_n(x,y)$, the $m$'s are the (possibly complex) roots of $p(y)=f_n(1,y)$; in the case of $g_n(x,y)$, the $m$'s are the (possibly complex) roots of $q(y)=g_n(1,y)$. Notice that $\theta=\gamma_2-\gamma_1$, where $\gamma_1$ is the angle formed by one of the ${\mathcal L}_i$'s and the positive $x$-axis, and $\gamma_2$ is the angle for the corresponding ${\mathcal M}_i$. Since $\gamma_2=\gamma_1+\theta$, denoting $m=\mbox{tg}(\gamma_1)$, $\tilde{m}=\mbox{tg}(\gamma_2)$, from the tangent angle addition formula we get that:
\[\tilde{m}=\frac{m+\omega}{1-m\omega}.\]Notice that $m\omega\neq 1$, because we can assume that we are relating two linear factors of $f_n(x,y)$ and $g_n(x,y)$ different from $x$. Now $\tilde{m}$ is a root of $q(y)$ and $m$ is a root of $p(y)$. Therefore, writing
\begin{equation}\label{elQ}
Q(t,y):=q\left(\frac{y+t}{1-yt}\right),
\end{equation}
$\omega$ is a root of
\begin{equation}\label{Pres}
P(t)=\mbox{Res}_y\left(p(y),\mbox{num}(Q(t,y))\right)=0,
\end{equation}
where ``num" denotes the numerator of the expression in brackets.
\end{itemize}

Continuing with the description of the algorithm, since we have $\bfb=\eta(\bfa,\overline{\bfa})$ with $\bfa=r\cdot (1+i\omega)$ and $\eta$ linear, we get $\bfb=\eta(r,\omega)$, $\overline{\bfb}=\overline{\eta}(r,\omega)$. By substituting these expressions into the system ${\mathcal S}$, we reach a bivariate system in the variables $r,\omega$. By adding the equation $P(\omega)=0$ to this system, the curves ${\mathcal C}_1, {\mathcal C}_2$ are similar if and only if this new system has a solution with $r\in {\Bbb R}$, $r\neq 0$ and $\omega\in \mathbb{R}$. 

One can wonder if the polynomial $P(t)$ in \eqref{Pres} can be identically zero, in which case the advantage of having a univariate equation is lost. The answer is affirmative, as shown by the following proposition. 

 \begin{proposition} \label{escero}
The polynomial $P(t)$ in \eqref{Pres} is identically zero iff $p(y)$ and $q(y)$ are both divisible by $y^2+1$.
\end{proposition} 

\begin{proof} 
Let
$$
p(y)=\prod\limits_{i=1}^n (y-m_i),\quad q(y)=\prod\limits_{i=1}^n (y-\tilde{m}_i).
$$
Hence
$$
Q(t,y)=\frac{1}{(1-yt)^n}\prod\limits_{i=1}^n (y + t -\tilde{m}_i +y\, t\, \tilde{m}_i ) ,
$$
which implies
$$
\mathrm{num}(Q(t,y)) = \prod\limits_{i=1}^n (y(1+  t\, \tilde{m}_i) +t-\tilde{m}_i  ) .
$$
Since $p(y)$ and $\mathrm{num}(Q(t,y))$ have both positive degree in $y$, $P(t)=0 $ if and only if they have a common factor with positive degree in $y$ (see Chapter 3 in \cite{Cox}). In turn, this happens if and only if  there exists $r\in\{1,\ldots,n\}$ such that $(y-m_r)$ divides $\mathrm{num}(Q(t,y))$. As a consequence $\mathrm{num}(Q(t,m_r ))=0$. Then there exists  $j\in\{1,\ldots,n\}$ with 
$(m_r + t -\tilde{m}_j +t \, m_r \, \tilde{m}_j )=0$, and so $m_r = \tilde{m}_j$, $1+\tilde{m}_j\,m_r=0$ hold. Since $p(y)$ and $q(y)$ are real polynomials, this implies that $y^2+1$ divides both polynomials. 
\end{proof}
%

Finally let us consider the case $\cos\theta=0$. Here we have $\bfa=i\mu$, where $\mu=|\bfa|\cdot \sin\theta$, and therefore $\overline{\bfa}=-i\mu$. By using $\bfb=\eta(\bfa,\overline{\bfa})$ and plugging these relationships into the system ${\mathcal S}$, we reach a polynomial system with just one variable, $\mu$. Hence, here we just need to check whether or not the $\gcd$ of the (univariate, and with complex coefficients) polynomials in the system has some nonzero root. 

Hence, we have the following algorithm {\tt Impl-Sim-General} to check similarity in the general case.

\begin{algorithm}
\begin{algorithmic}[1]
\REQUIRE Two irreducible, implicit algebraic curves $f(x,y)$, $g(x,y)$, neither lines nor circles, satisfying the hypotheses of the general case.
\ENSURE Whether or not they are related by an orientation-preserving similarity.
\STATE Find the system ${\mathcal S}$.
\STATE Find $j$ such that $\alpha_{n-j,j}\neq 0$ and $\Delta_j\neq 0$. 
\STATE Find $\lambda(\bfa,\overline{\bfa})$ from \eqref{eq:lamda}. 
\STATE Find $\bfb=\eta(\bfa,\overline{\bfa})$ $\bar{\bfb}=\bar{\eta}(\bfa,\overline{\bfa})$ and from \eqref{eq:first} and \eqref{eq:second}.
\STATE  Substitute $\lambda:=\lambda(\bfa,\overline{\bfa})$, $\bfb:=\eta(\bfa,\overline{\bfa})$, $\overline{\bfb}=\overline{\eta}(\bfa,\overline{\bfa})$ in the system ${\mathcal S}$. 

[Similarities with $\cos\theta\neq 0$:] .
\STATE Compute the polynomial $P(t)$.
\STATE  Substitute $\bfa=r\cdot (1+i\omega)$ , $\overline{\bfa}=r\cdot (1-i\omega)$ , $\bfb=\eta(\bfa,\overline{\bfa})$, $\overline{\bfb}=\overline{\eta}(\bfa,\overline{\bfa})$ in ${\mathcal S}$. If $P(t)$ is not identically zero, add the equation $P(\omega)=0$ to the system.
\STATE Check if the system has any solution with $r\in {\Bbb R}$, $r\neq 0$ and $\omega \in {\Bbb R}$.
\STATE In the affirmative case, return ``{\tt The curves are related by an orientation-preserving similarity}".

 [Similarities with $\cos\theta=0$.]
\STATE Substitute $\bfa=i\mu$, $\overline{\bfa}=-i\mu$, $\bfb=\eta(\bfa,\overline{\bfa})$, $\overline{\bfb}=\overline{\eta}(\bfa,\overline{\bfa})$ in ${\mathcal S}$.
\STATE Check if the univariate polynomials in $\mu$ obtained this way have a non-trivial $\gcd$ with some nonzero real root.
\STATE In the affirmative case, return ``{\tt The curves are related by an orientation-preserving similarity}".
\STATE If no similarities have been found, return ``{\tt The curves are not related by an orientation-preserving similarity}".
\end{algorithmic}
\caption*{{\bf Algorithm} {\tt Impl-Sim-General}}
\end{algorithm}

\section{The special case.} \label{sec-special}

From Lemma \ref{lem-no-gen-case} in Section \ref{sec-main}, here we have $\bfj=0$. Furthermore, $\alpha_{n,0}\neq 0$ and $\beta_{n,0}\neq 0$. Now if $\alpha_{n-1,0}\neq 0$, then from \eqref{eq:first} with $j=0$ we can write $\bfa$ as a linear function of $\bfb,\overline{\bfb}$, i.e. $\bfa=\rho(\bfb,\overline{\bfb})$. Let us see that this requirement can always be fulfilled by means of an appropriate translation.  

\begin{lemma} \label{lem-trans}
Let $T(z)=z+\kappa$ represent a translation by $\kappa\in {\Bbb C}$, and let $\tilde{\mathcal C}_1=T({\mathcal C}_1)$. Also, let $j\in \{0,\ldots,n-1\}$. The coefficient of $\tilde{\mathcal C}_1$ of bidegree $(n-j-1,j)$ is
\begin{equation}\label{beta}
\tilde{\alpha}_{n-j-1,j}=\alpha_{n-j-1,j}+\kappa\cdot (n-j)\alpha_{n-j,j}+\overline{\kappa}\cdot (j+1)\alpha_{n-j-1,j+1}.\end{equation}As a consequence, if $\alpha_{n-j,j}\neq 0$ and $\alpha_{n-j-1,j}=0$ then for almost all $\kappa\in {\Bbb C}$, we have $\tilde{\alpha}_{n-j-1,j}\neq 0$.
\end{lemma}

\begin{proof} The coefficient of $\tilde{\mathcal C}_1=T({\mathcal C}_1)$ of bidegree $(n-j-1,j)$ comes from the terms of $F$ with bidegrees $(n-j,j)$, $(n-j-1,j)$ and $(n-j-1,j+1)$. Each of these terms provides, in turn, the terms of \eqref{beta}: the term of $F$ of bidegree $(n-j,j)$ provides the first term of \eqref{beta}, the term of bidegree $(n-j-1,j)$ provides the term in $\kappa$, and the term of $(n-j-1,j+1)$ provides the term in $\overline{\kappa}$. As for the second part of the statement, assume that $\alpha_{n-j,j}\neq 0$, $\alpha_{n-j-1,j}=0$, and suppose by contradiction that $\tilde{\alpha}_{n-j-1,j}$ vanishes for all $\kappa$. By substituting $\kappa=1$ in \eqref{beta}, we have $(n-j)\alpha_{n-j,j}+(j+1)\alpha_{n-j-1,j+1}=0$; by substituting $\kappa=i$ in \eqref{beta}, we get that $(n-j)\alpha_{n-j,j}-(j+1)\alpha_{n-j-1,j+1}=0$. Putting these two equations together, we deduce that $\alpha_{n-j-1,j+1}=\alpha_{n-j,j}=0$, contradicting that $\alpha_{n-j,j}\neq 0$.
\end{proof}

Since similarities form a group, ${\mathcal C}_1$ and ${\mathcal C}_2$ are similar if and only if $\tilde{\mathcal C}_1=T({\mathcal C}_1)$ and ${\mathcal C}_2$ are similar. Therefore, there is no problem in applying a translation $T$ on ${\mathcal C}_1$. Furthermore, by Lemma \ref{lem-trans} a random translation will suffice. Now after replacing ${\mathcal C}_1$ by $T({\mathcal C}_1)$, we can write $\bfa=\xi(\bfb,\overline{\bfb})$, where $\xi$ is linear. By substituting $\bfa=\xi(\bfb,\overline{\bfb})$ into the system ${\mathcal S}$, we are led to a bivariate polynomial system ${\mathcal S}^{\star}$ with complex coefficients, where the unknowns are $\bfb$, $\overline{\bfb}$. In turn, this gives rise to a real bivariate polynomial system where the unknowns are the real and imaginary parts $b_1,b_2$ of $\bfb$. Then the curves are similar if and only if this system has a real solution where $\bfa\neq 0$.

Hence, we have the following algorithm {\tt Impl-Sim-Spec} to check similarity in this case.

\begin{algorithm}
\begin{algorithmic}[1]
\REQUIRE Two irreducible, implicit algebraic curves $f(x,y)$, $g(x,y)$, neither lines nor circles, satisfying the hypotheses of the special case.
\ENSURE Whether or not they are related by an orientation-preserving similarity.
\STATE Check whether or not $\alpha_{n-1,0}\neq 0$.
\STATE If $\alpha_{n-1,0}=0$ then find a translation $T(z)=z+k$ such that $\tilde{\alpha}_{n-1,1}\neq 0$, and replace ${\mathcal C}_1\to T({\mathcal C}_1)$.
\STATE Write $\bfa=\xi(\bfb,\overline{\bfb})$, $\lambda(\bfa,\overline{\bfa})$.
\STATE Substitute the above expressions into ${\mathcal S}$.
\STATE Write $\bfb=b_1+ib_2$ to get the real bivariate polynomial system ${\mathcal S}^{\star}$.
\STATE Check if ${\mathcal S}^{\star}$ has a real solution with $\bfa\neq 0$.
\STATE In the affirmative case, return ``{\tt The curves are related by an orientation-preserving similarity}". Otherwise, return ``{\tt The curves are not related by an orientation-preserving similarity}".
\end{algorithmic}
\caption*{{\bf Algorithm} {\tt Impl-Sim-Spec}}
\end{algorithm}

\section{Observations on the ``orientation-reversing" case.}\label{sec-or-reverse}

In order to look for orientation-reversing similarities transforming ${\mathcal C}_1$ into ${\mathcal C}_2$, we must check the consistency of the system ${\mathcal S}'$ (see the beginning of Section \ref{sec-main}), which stems from Equation \eqref{eq:eqimplicits-2}. The analysis is very similar to orientation-preserving similarities. However, for completeness we summarize here some observations that must be taken into account. 

First, Equation \eqref{eq:lamda} must be replaced by 
\begin{equation}\label{eq:lamda2}
\lambda=\frac{\beta_{n-j,j}\bfa^{n-j}\overline{\bfa}^{j}}{ {\alpha}_{ j,n-j}},\end{equation}
which is deduced by comparing the terms of $z^{j}\bar{z}^{n-j}$ at both sides of \eqref{eq:eqimplicits-2}. Similarly, Equation \eqref{eq:first} must be replaced by 

\begin{equation}\label{eq:first2}
(n-j)\beta_{n-j,j}\cdot \bfb+(j+1)\beta_{n-j-1,j+1}\cdot \overline{\bfb}+\beta_{n-j-1,j}= \frac{\beta_{n-j,j}\cdot \alpha_{j,n-j-1}}{ {\alpha}_{j,n-j}}\cdot \bfa, 
\end{equation}
and \eqref{eq:second} must be replaced by 

\begin{equation}\label{eq:second2}
(j+1)\overline{\beta}_{n-j-1,j+1}\cdot \bfb+(n-j)\overline{\beta}_{n-j,j}\cdot\overline{\bfb}+\overline{\beta}_{n-j-1,j}= \frac{\overline{\beta}_{n-j,j}\cdot\overline{\alpha}_{j,n-j-1}}{\overline{\alpha}_{j,n-j}}\cdot \overline{\bfa}.
\end{equation}

Lemma \ref{lem-no-gen-case} also holds here, and again we must distinguish between the general case and the special case, which are formulated exactly in the same way. There are some slight changes in the discussion at the beginning of Section \ref{sec-general}, though. In this case, the curve defined by $f_n(x,y)=0$ is transformed into the curve defined by $g_n(x,y)=0$ under the transformation $h(z)=\bfa \bar{z}$, which geometrically is equivalent to a symmetry with respect to the $x$-axis, followed by a rotation about the origin by an angle $\theta$. Again, the case $\cos(\theta)=0$ must be considered separately, and leads to a simpler case, which presents no new difficulties. The case (1) in Section \ref{sec-general} leads to  either $P(t)=t$, or $P(t)=t(t-d/c)$ or $P(t)=t-d/c$. In case (2) , we have $\theta=\gamma_1+\gamma_2$. So denoting $m=\mbox{tg}(\gamma_1)$, $\tilde{m}=\mbox{tg}(\gamma_2)$, and since $\gamma_2=\theta-\gamma_1$, we get that:
\[\tilde{m}=\frac{\omega-m}{1+m\omega}.\]Hence, $\omega$ is a root of
\[
P(t)=\mbox{Res}_y\left(p(y),\mbox{numer}\left(q\left(\frac{t-y}{1+yt}\right)\right)\right)=0.
\]
Furthermore, Proposition \ref{escero} also holds here. 
 Additionally, in the special case, Lemma \ref{lem-trans} is applied to $\alpha_{j,n-j-1}$.

\section{Implementation and timings }\label{sec-impl}

Algorithms \texttt{Impl-Sim-General} and \texttt{Impl-Sim-Spec} have been implemented in Maple 2015 (\cite{maple}). The code is available in \cite{gmdt}. Since we used Maple's generic solver \texttt{SolveTools:-PolynomialSystem}, the algorithm is deterministic.

\subsection{Example I}
Let $\CCC_1,\CCC_2$ be defined by $f(x,y)$ and $g(x,y)$, where
$$
f(x,y)=15\,{x}^{2}y-40\,x{y}^{2}-15\,{y}^{3}+5\,{x}^{2}+5\,xy-35\,{y}^{2}+5\,
x-5\,y+2
$$
and
$$
g(x,y)={y}^{3}+2\,x{y}^{2}-{x}^{2}y-xy-2\,{x}^{3}+1.
$$
Both curves are shown in Figure \ref{fig:curvas_al}.
\begin{figure}
\centering
\begin{subfigure}[b]{0.40\textwidth}
  \centering
  \includegraphics[scale=0.33]{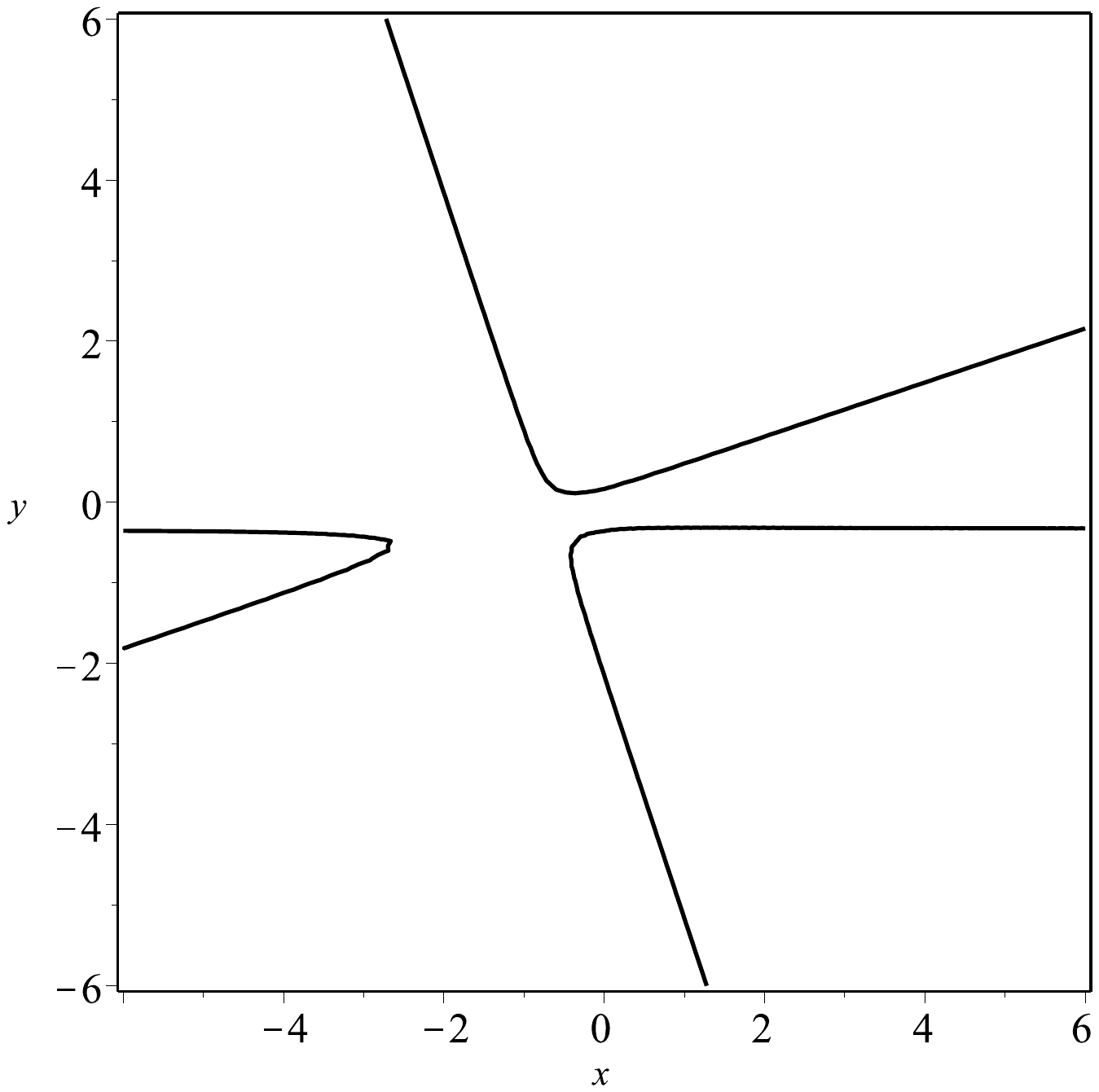}
\end{subfigure}%
\quad
\begin{subfigure}[b]{0.40\textwidth}
  \centering
  \includegraphics[scale=0.33]{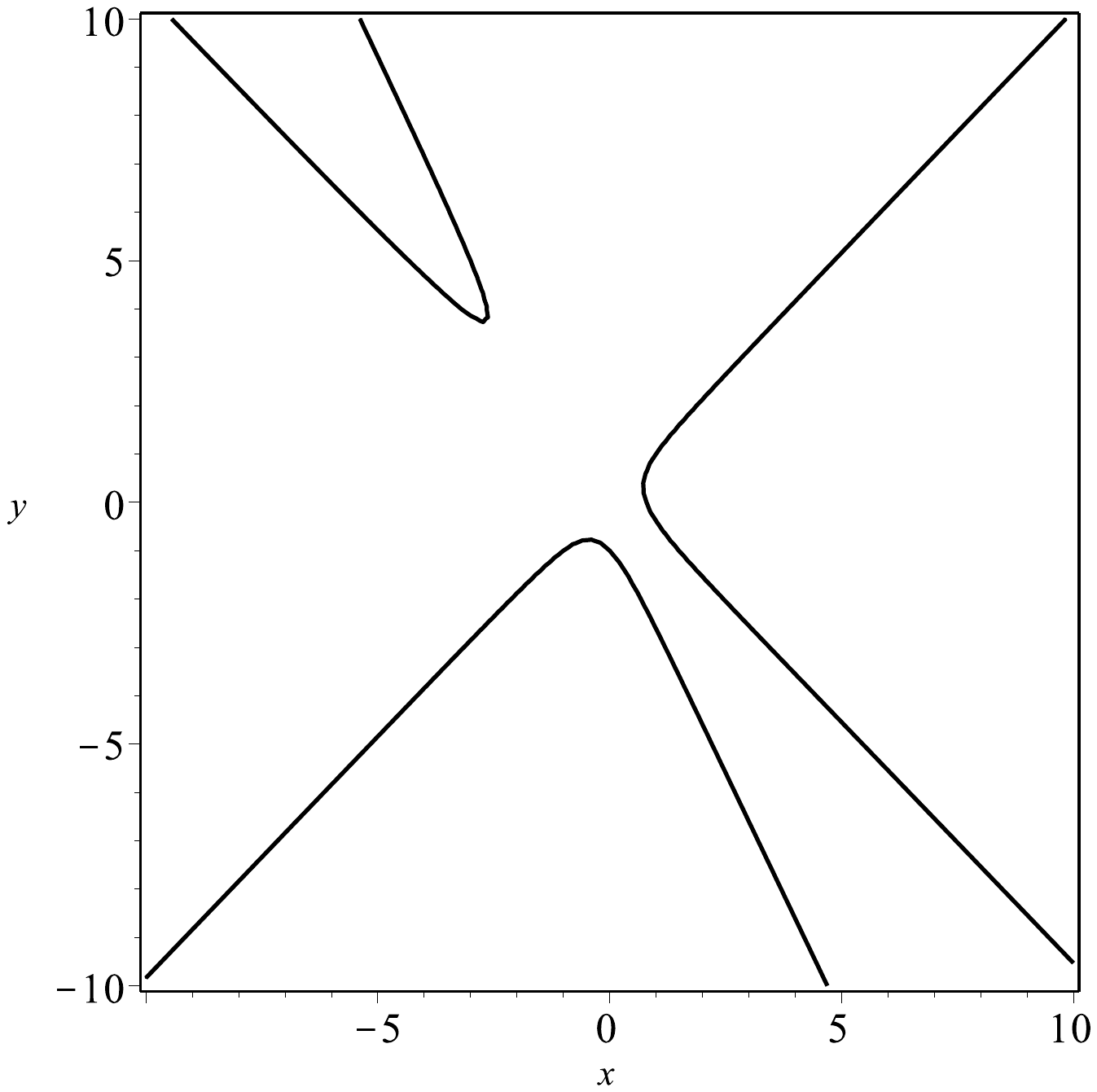}
\end{subfigure}%
\vspace{-3cm}
\caption{The curve $\CCC_1$ (left) and a related curve $\CCC_2$ obtained by similarity (right).}\label{fig:curvas_al} 
\end{figure}
Using Lemma \ref{lem-no-gen-case}, one can check that we fall into the general case. By applying the algorithm {\tt Impl-Sim-General}, we get an orientation-preserving similarity $h(z)=\bfa z +\bfb$ with 
$$
P(t)=-125\,(t^3+2\,t^2-t-2)\,(448\,t^6+4416\,t^5+8880\,t^4-1920\,t^3-8880\,t^2+4416\,t-448).
$$
Adding the equation $P(\omega)=0$ to the system in $r,\omega$ obtained from ${\mathcal S}$, we get
$$
\omega=-2,r=1, \lambda= \left( -{\frac {11}{125}}-{\frac {2\,i}{125}} \right) {r}^{3} \left( 1+i\omega \right) ^{3}=1,
$$
$$
\bfa=r(1+ i \, \omega )= 1-2i,\,\,\bfb= {\frac {17\,r}{50}}-{\frac {19\,\omega\,
r}{50}}-\frac{1}{10}+({\frac {71\,\omega\,r}{100}}+{\frac {11\,r}{50}}+\frac{1}{5} )i= 1-i.
$$
The ratio of this similarity is equal to $\sqrt{5}.$

\subsection{Example II}
Let $\CCC_1,\CCC_2$ be defined by $f(x,y)$ and $g(x,y)$, where
$$
f(x,y)={x}^{4}+2\,{x}^{2}{y}^{2}+{y}^{4}-8\,{x}^{2}y-8\,{y}^{3}+12\,{x}^{2}-6
\,xy+20\,{y}^{2}+12\,x-16\,y;
$$
and
$$
g(x,y)=2\,{x}^{4}+4\,{x}^{2}{y}^{2}+2\,{y}^{4}-{x}^{2}+{y}^{2}.
$$
Both curves are shown in Figure \ref{fig:lemniscatas}. 
\begin{figure}
\centering
\begin{subfigure}[b]{0.40\textwidth}
  \centering
  \includegraphics[scale=0.33]{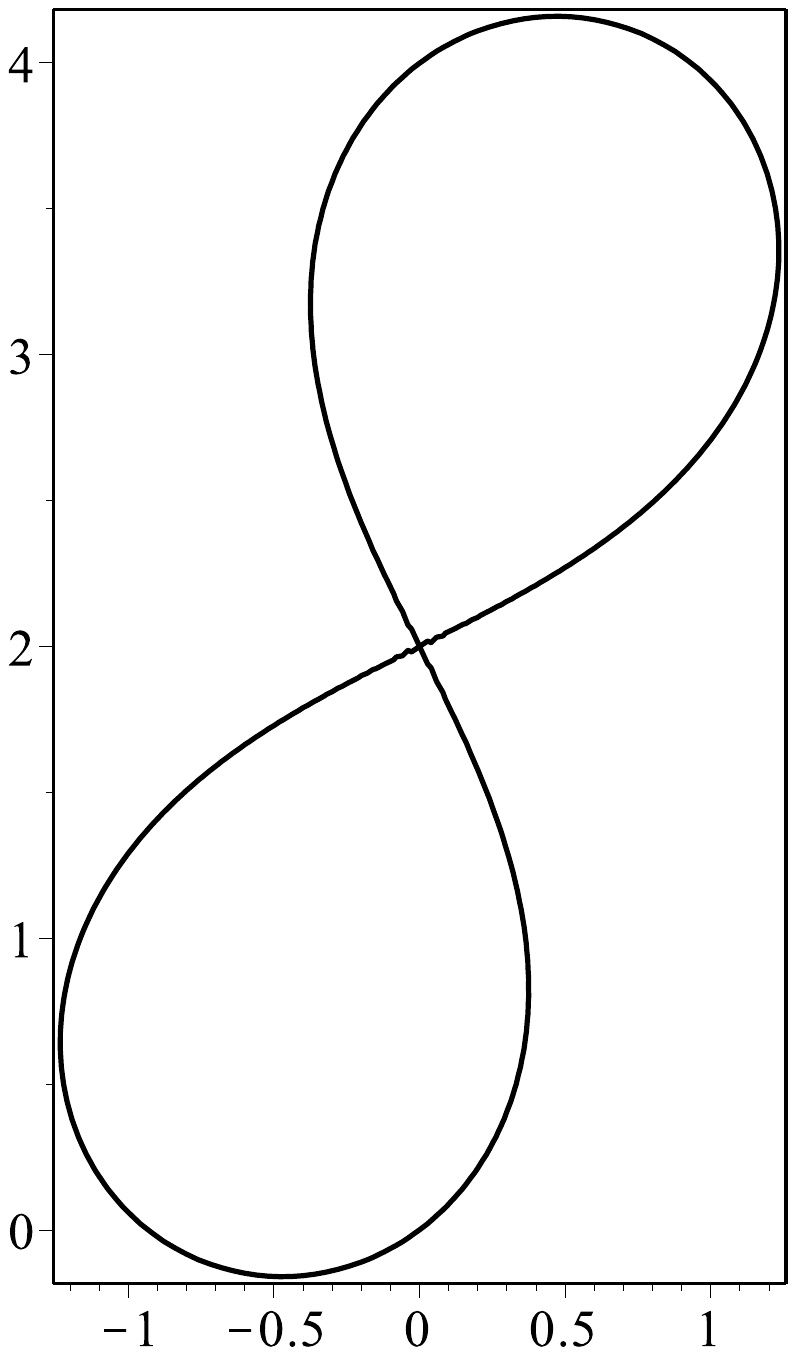}
\end{subfigure}%
\quad
\begin{subfigure}[b]{0.40\textwidth}
  \centering
  \includegraphics[scale=0.33]{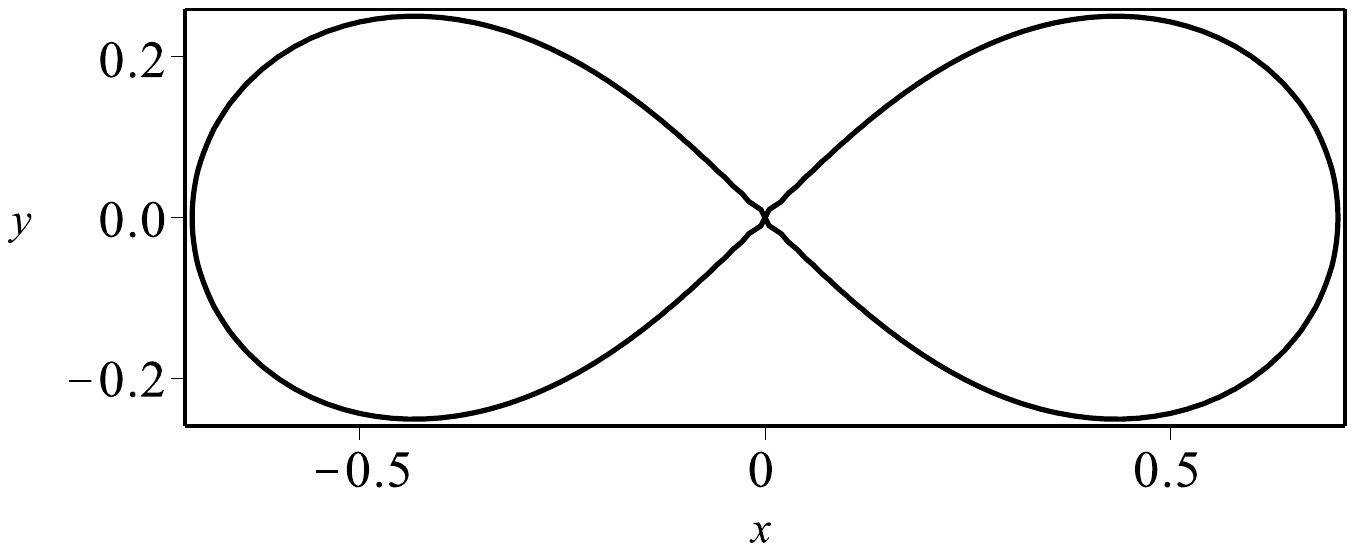}
\end{subfigure}%
\vspace{-3cm}
\caption{A lemniscate  $\CCC_1$ (left) and a related lemniscate $\CCC_2$ obtained by similarity (right).}\label{fig:lemniscatas}
\end{figure}
In this case we also fall into the general case. However, here $P(t)$ is identically zero. By applying the algorithm {\tt Impl-Sim-General}, we first find two orientation-preserving similarities $h(z)=\bfa z +\bfb$. In terms of $\omega$ and $r$, we get $\bfb= -2\,ir \left( 1+i\omega \right)$ and $\lambda=2\,{r}^{4} \left( 1+i\omega \right) ^{2} \left( 1-i\omega \right) ^{2}$. The first orientation-preserving similarity corresponds to $\omega= - 3$ and $r= 1/10$. Since $\bfa=r(1+ i \, \omega )$,  we have
$$
\bfa = \frac {1}{10}-\frac {3}{10}\,i , \quad \bfb=-\frac {3}{5}- \frac {1}{5}\,i, \quad \lambda=\frac {1}{50}.
$$
The second orientation-preserving similarity corresponds to $\omega=-3$ and $r= -1/10$, and
$$
\bfa = -\frac {1}{10}+\frac {3}{10}\,i , \quad \bfb= \frac {3}{5}+ \frac {1}{5}\,i, \quad \lambda=\frac {1}{50}.
$$
Additionally we find two orientation-reversing similarities $h(z)=\bfa \bar{z} +\bfb$. In terms of $\omega$ and $r$, we have $\bfb= 2\,ir \left( 1+i\omega \right)$ and $\lambda=2\,{r}^{4} \left( 1+i\omega \right) ^{2} \left( 1-i\omega \right) ^{2}$. The first orientation-reversing similarity corresponds to $\omega=  3$ and $r= 1/10$, and 
$$
\bfa = \frac {1}{10} + \frac {3}{10}\,i , \quad \bfb=-\frac {3}{5} + \frac {1}{5}\,i, \quad \lambda=\frac {1}{50}.
$$
The second orientation-reversing similarity corresponds to $\omega= 3$ and $r= -1/10$, and 
$$
\bfa = -\frac {1}{10}-\frac {3}{10}\,i , \quad \bfb= \frac {3}{5} - \frac {1}{5}\,i, \quad \lambda=\frac {1}{50}.
$$
The ratio of all these similarities is equal to $\frac {1}{\sqrt{10}}.$

\subsection{Example III}

Let $\CCC_1,\CCC_2$ be defined by $f(x,y)$ and $g(x,y)$, where
$$
f(x,y)=19\,{x}^{3}+90\,{x}^{2}y-18\,x{y}^{2}+35\,{y}^{3}+51\,{x}^{2}+237\,xy-
90\,{y}^{2}+39\,x+195\,y-1,
$$
and
$$
g(x,y)={x}^{3}+{y}^{3}-3\,yx  \text{ (Descartes' Folium)}.
$$
Both curves are shown in Figure \ref{fig:curvas_fol}.
\begin{figure}
\centering
\begin{subfigure}[b]{0.40\textwidth}
  \centering
  \includegraphics[scale=0.33]{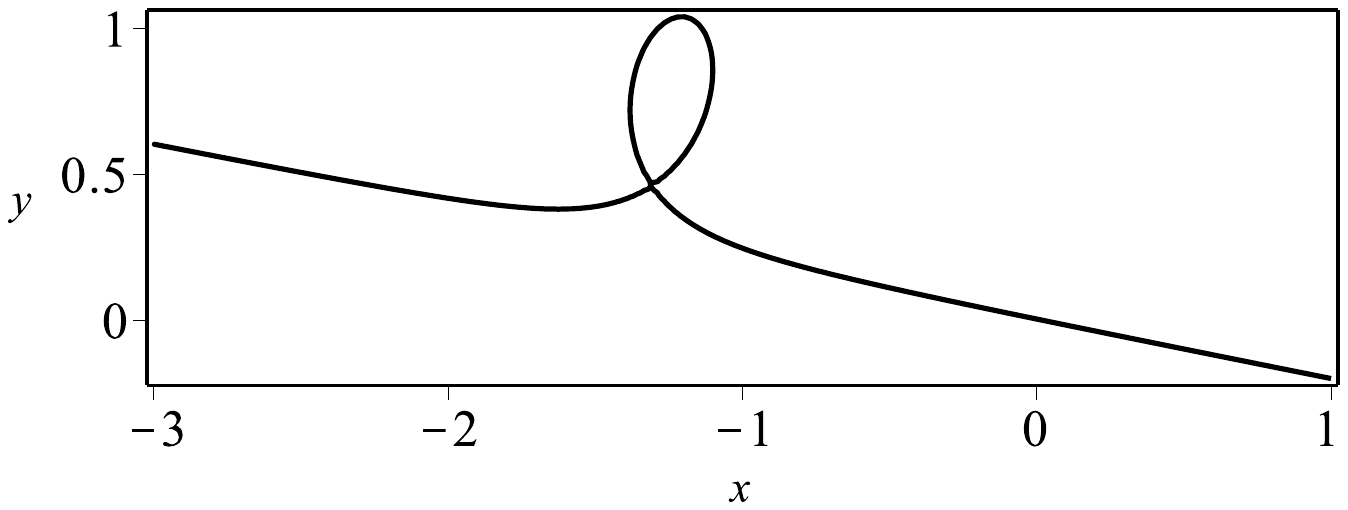}
\end{subfigure}%
\quad
\begin{subfigure}[b]{0.40\textwidth}
  \centering
  \includegraphics[scale=0.33]{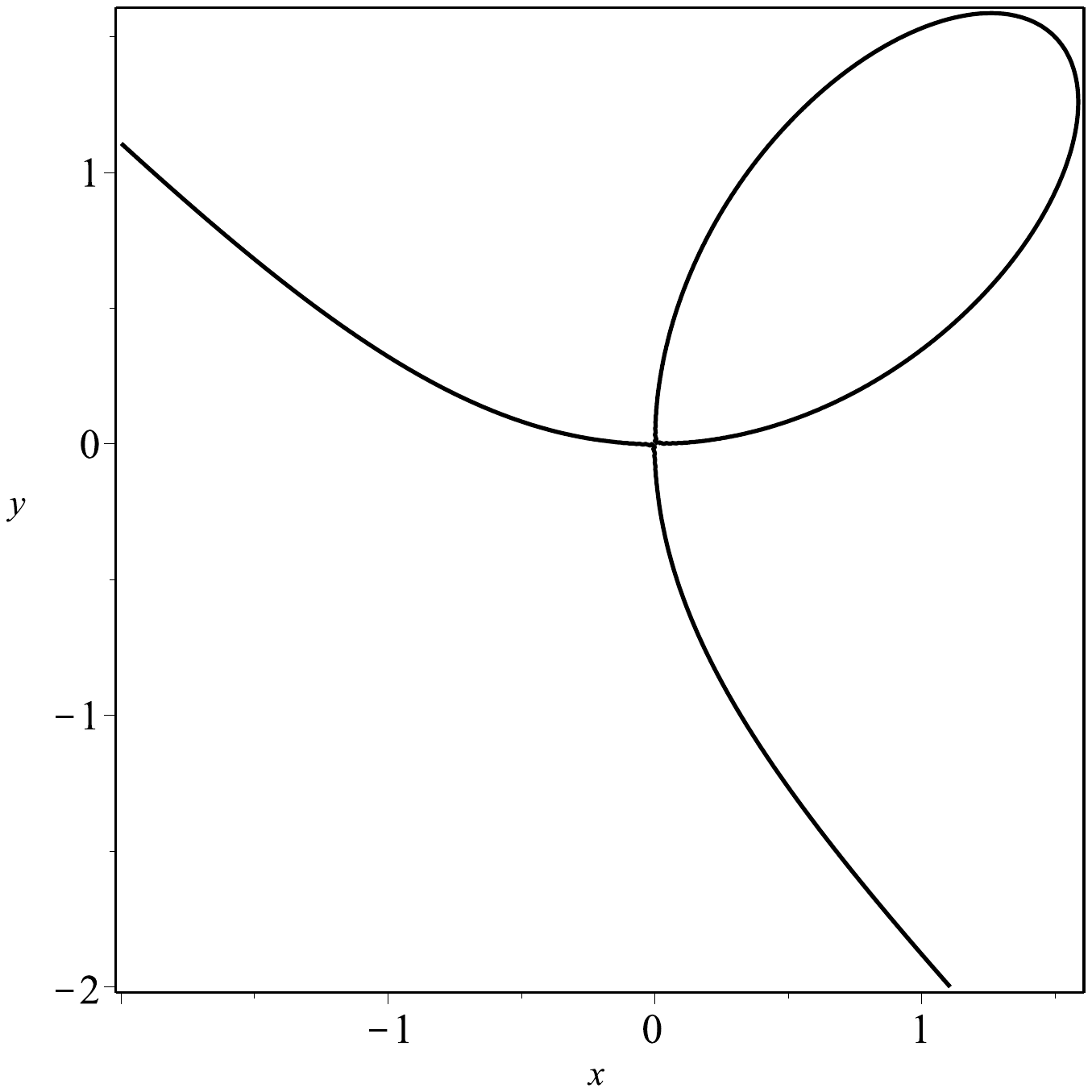}
\end{subfigure}%
\vspace{-3cm}
\caption{The curve $\CCC_1$ (left) and a related curve $\CCC_2$ obtained by similarity (right).}\label{fig:curvas_fol} 
\end{figure}
Using Lemma \ref{lem-no-gen-case}, we observe that we fall into the special case addressed in Section \ref{sec-special}. By applying the algorithm {\tt Impl-Sim-Spec}, we first find an orientation-preserving similarity $h(z)=\bfa z +\bfb$, where $\bfb = 3-4i$, $\bfa = 3-2\,i$, $\lambda=1$. Additionally, we find an orientation-reversing similarity $h(z)=\bfa \bar{z} +\bfb$, where $\bfb= -4+3i $, $\bfa = -2+3\,i$, $\lambda=1$.
 %
The ratio of these similarities is equal to $ \sqrt{13}.$

\subsection{Performance}
Algorithms \texttt{Impl-Sim-General} and \texttt{Impl-Sim-Spec} can be unified into one algorithm that first checks the case we are in, and then executes Algorithms \texttt{Impl-Sim-General} or \texttt{Impl-Sim-Spec} accordingly. We tested the performance of this algorithm on curves of various degrees $d$ and bitsize $\tau$\footnote{The \emph{bitsize} $\tau := \lceil \log_2 k \rceil + 1$ of an integer $k$ is the number of bits needed to represent it. When we refer to the bitsize of an implicit algebraic curve, we imply the maximum of the bitsizes of the absolute values of the coefficients of the implicit equation.} More precisely, for every $d \in \lrb{5..10}$ and $\tau=2^i, i  \in  \lrb{ 0..5}$, we generated 20 random dense bivariate polynomials $g(x,y)$ of degree $d$ and bitsizes bounded by $\tau$. Once generated, we computed $F(z,\bar{z}):=G(\bfa z+\bfb,\bfa\bar{ z}+\bfb)$, with $(\bfa,\bfb)$  a pair of random complex numbers with coefficients in the range $\lrb{-10..10}$. Finally we tested the algorithms with $F(z,\bar{z})$ and $G(z,\bar{z})$. Observe that even if the bitsize of $g(x,y)$ is bounded by $\tau$, the bitsize of $F(z,\bar{z})$ is usually bigger than $\tau$. The code, as well as the random inputs, are given in \cite{gmdt}; the timings are provided in Table \ref{tab:performance_gen}.

\begin{table}
\begin{center}
\begin{tabular*}{\columnwidth}{l@{\extracolsep{\stretch{1}}}*{6}{r}@{}}
\toprule
  t  &  $\tau=1$ &  $\tau=2$ &  $\tau=4$ &  $\tau=8$ & $\tau=16$ & $\tau=32$ \\ 
\midrule
$d=5$  &      0.70 &      0.73 &      0.73 &      0.79 &     0.82 &      2.73 \\ 
$d=6$  &      1.04 &      1.11 &      1.22 &     1.51 &    5.37 &    9.11 \\ 
$d=7$  &      1.84 &     3.24 &     8.10 &    11.10 &    12.76&   14.71 \\ 
$d=8$  &     15.94&     17.80&    20.40 &   21.91 &  25.23 &       32.40    \\ 
$d=9$  &     24.82&    26.86 &   27.79 &  30.75 &       36.20    &    42.15       \\ 
$d=10$  &    35.84 &   38.60 &  41.81 &    47.45       &   52.13        &      72.27     \\ 
\bottomrule
\end{tabular*}
\end{center}
\caption{Average CPU time t (seconds) of the algorithms applied to random dense bivariate polynomials.}\label{tab:performance_gen}
\end{table}

After examining the correlation matrix of the data we observed that the degree highly influences the timing. The influence of the bitsize is weaker. By using Maple's command \texttt{Statistics:- NonlinearFit}, we compared several following multivariate model functions and found that the following function fits our data quite well,  
\begin{eqnarray*}
t&=&- 0.0987387345679027195\,{d}^{3}+ 3.89759623015877255\,{d}^{2}\\
&&+ 0.180291478322672782\,\tau\,d - 33.7316950021062496\,d\\
&&- 0.853707190239283653\,\tau+ 84.1154055239685192.\\
 \end{eqnarray*}
Figure  \ref{grado}  shows the behavior of this model, and compares it with the data collected in our experiments. Notice that this model suggests a time complexity ${\mathcal O}(d^3+d\tau)$ for the algorithm. 
 \begin{figure}[htbp]
\begin{center}
\centerline{$\begin{array}{rl}
  \includegraphics[scale=0.45]{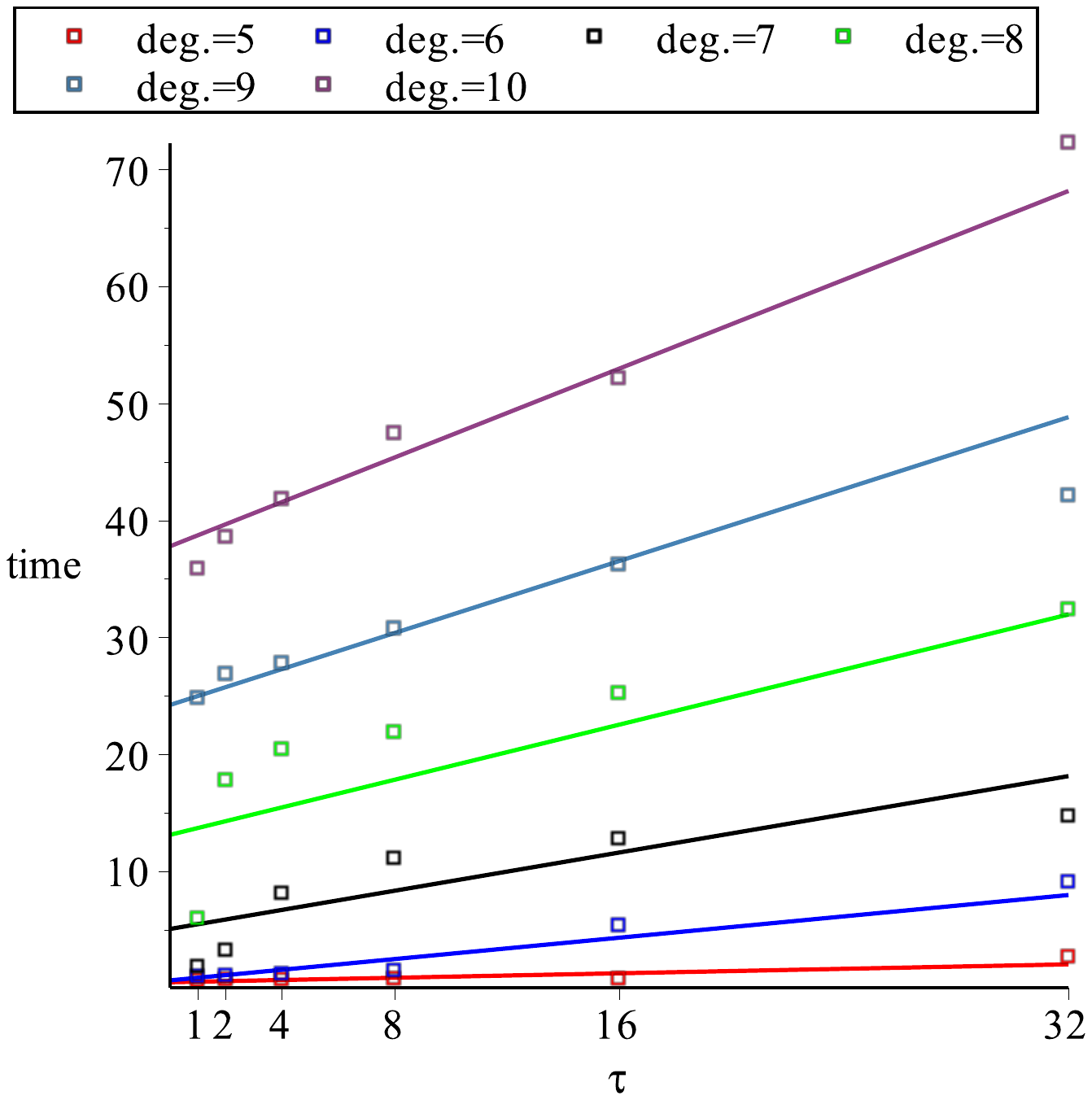} & \hspace{-1cm}\includegraphics[scale=0.45]{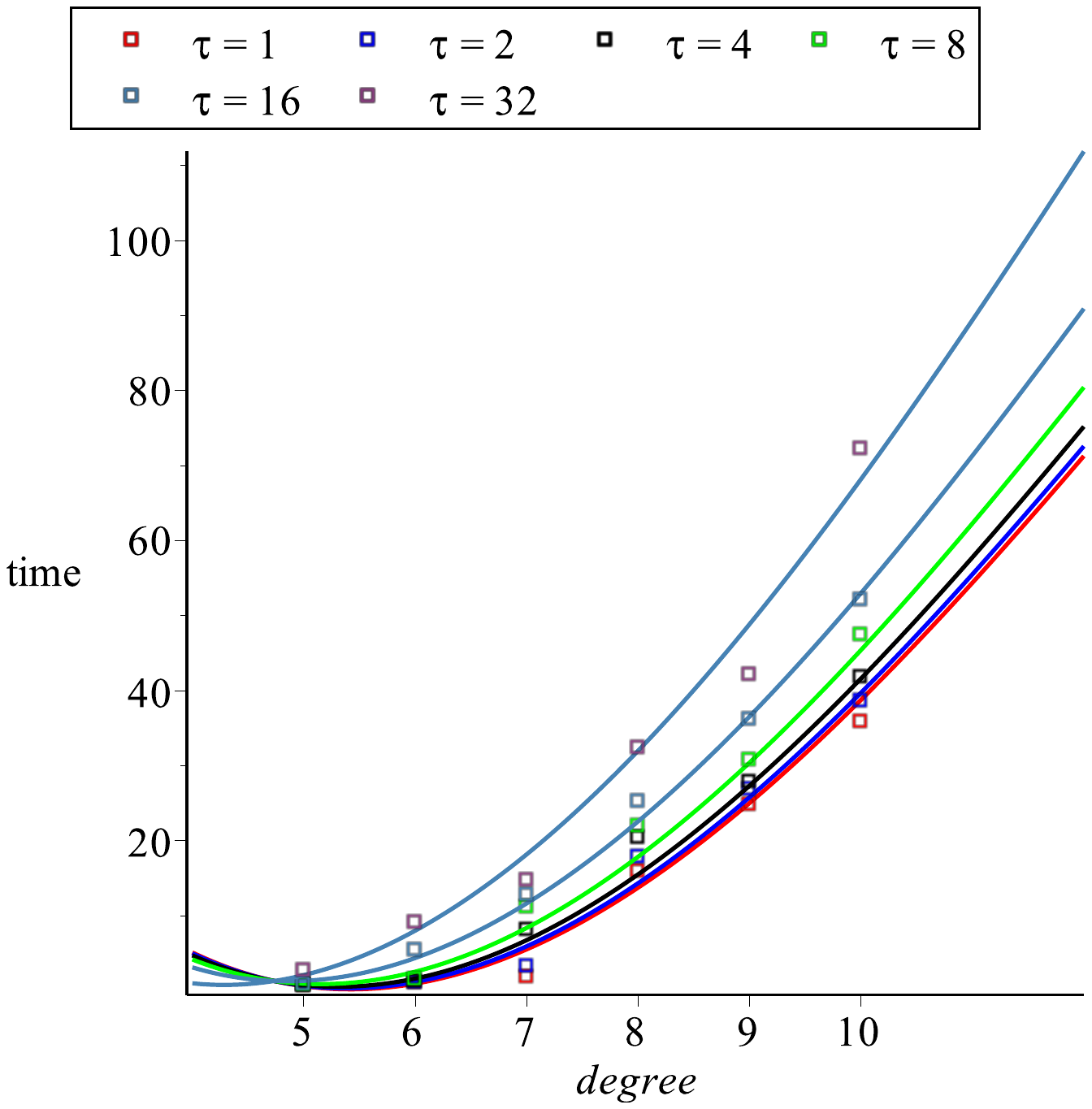} 
  \end{array}$}
  \vspace{-4cm}
\caption{Degrees, bitsizes and timings.}\label{grado}
\end{center}
\end{figure}

We must emphasize that these timings corresponds to \emph{dense} polynomials. Therefore, in practice the performance is better than Table \ref{tab:performance_gen} predicts. In order to illustrate this statement we tested the algorithms with several classical curves, which we list below:
 
\begin{itemize}
\item \emph{Folium of Descartes},
\[ g(x,y)= {x}^{3}-3\,yx+{y}^{3}, \]
\item \emph{Lemniscate of Bernoulli},
\[
g(x,y)=2\,{x}^{4}+4\,{x}^{2}{y}^{2}+2\,{y}^{4}-{x}^{2}+{y}^{2},
\]
\item \emph{Epitrochoid},
\[
g(x,y)={x}^{4}+2\,{x}^{2}{y}^{2}+{y}^{4}-34\,{x}^{2}-34\,{y}^{2}+96\,x-63,
 \]
\item \emph{Offset curve to the cardioid},
\begin{eqnarray*}
g(x,y)&=&{x}^{8}+4\,{x}^{6}{y}^{2}+6\,{x}^{4}{y}^{4}+4\,{x}^{2}{y}^{6}+{y}^{8}+
16\,{x}^{6}y+48\,{x}^{4}{y}^{3}\\
&&+48\,{x}^{2}{y}^{5}+16\,{y}^{7}-140\,{x
}^{6}-324\,{x}^{4}{y}^{2}\\
&&-228\,{x}^{2}{y}^{4}-44\,{y}^{6}-1552\,{x}^{4
}y-2848\,{x}^{2}{y}^{3}\\
&&-1296\,{y}^{5}+2416\,{x}^{4}-864\,{x}^{2}{y}^{2
}-3024\,{y}^{4}+28800\,{x}^{2}y\\
&&+21888\,{y}^{3}+43776\,{x}^{2}+124416\,
{y}^{2}+228096\,y+145152,
\end{eqnarray*}
\item \emph{Hypocycloid},
 \begin{eqnarray*}
g(x,y)&=&{x}^{8}+4\,{x}^{6}{y}^{2}+6\,{x}^{4}{y}^{4}+4\,{x}^{2}{y}^{6}+{y}^{8}12\,{x}^{6}+36\,{x}^{4}{y}^{2}+36\,{x}^{2}{y}^{4}\\
&&+
+12\,{y}^{6}-512\,{x}
^{5}+5120\,{x}^{3}{y}^{2}-2560\,x{y}^{4}+150\,{x}^{4}\\
&&+300\,{x}^{2}{y}^
{2}+150\,{y}^{4}+67500\,{x}^{2}+67500\,{y}^{2}-759375,
 \end{eqnarray*}

\item \emph{4-leaf Rose},
\[
g(x,y)={x}^{6}+3\,{x}^{4}{y}^{2}+3\,{x}^{2}{y}^{4}+{y}^{6}-{x}^{4}+2\,{x}^{2}
{y}^{2}-{y}^{4},
\]
\item \emph{8-leaf Rose},
 \begin{eqnarray*}
g(x,y)&=&{x}^{10}+5\,{x}^{8}{y}^{2}+10\,{x}^{6}{y}^{4}+10\,{x}^{4}{y}^{6}+5\,{x
}^{2}{y}^{8}+{y}^{10}-{x}^{8}+12\,{x}^{6}{y}^{2}\\
&&-38\,{x}^{4}{y}^{4}+12
\,{x}^{2}{y}^{6}-{y}^{8},
 \end{eqnarray*}
\item \emph{12-leaf Rose},
 \begin{eqnarray*}
g(x,y)&=&{x}^{14}+7\,{x}^{12}{y}^{2}+21\,{x}^{10}{y}^{4}+35\,{x}^{8}{y}^{6}+35
\,{x}^{6}{y}^{8}+21\,{x}^{4}{y}^{10}\\
&&+7\,{x}^{2}{y}^{12}+{y}^{14}-{x}^{12}+30\,{x}^{10}{y}^{2}-255\,{x}^{8}{y}^{4}+452\,{x}^{6}{y}^{6}\\
&&-255\,{x}^{4}{y}^{8}+30\,{x}^{2}{y}^{10}-{y}^{12},
 \end{eqnarray*}
\item \emph{16-leaf Rose},
 \begin{eqnarray*}
g(x,y)&=&{x}^{18}+9\,{x}^{16}{y}^{2}+36\,{x}^{14}{y}^{4}+84\,{x}^{12}{y}^{6}+
126\,{x}^{10}{y}^{8}+126\,{x}^{8}{y}^{10}\\
&&+84\,{x}^{6}{y}^{12}+36\,{x}^{4}{y}^{14}+9\,{x}^{2}{y}^{16}+{y}^{18}-{x}^{16}+56\,{x}^{14}{y}^{2}\\
&&-924\,{x}^{12}{y}^{4}+3976\,{x}^{10}{y}^{6}-6470\,{x}^{8}{y}^{8}+3976\,
{x}^{6}{y}^{10}-924\,{x}^{4}{y}^{12}\\
&&+56\,{x}^{2}{y}^{14}-{y}^{16},
\end{eqnarray*}
\item  \emph{20-leaf Rose},
 \begin{eqnarray*}
g(x,y)&=&{x}^{22}+11\,{x}^{20}{y}^{2}+55\,{x}^{18}{y}^{4}+165\,{x}^{16}{y}^{6}+
330\,{x}^{14}{y}^{8}+462\,{x}^{12}{y}^{10}\\
&&+462\,{x}^{10}{y}^{12}+330\,
{x}^{8}{y}^{14}+165\,{x}^{6}{y}^{16}+55\,{x}^{4}{y}^{18}+11\,{x}^{2}{y
}^{20}+{y}^{22}\\
&&-{x}^{20}+90\,{x}^{18}{y}^{2}-2445\,{x}^{16}{y}^{4}+
19320\,{x}^{14}{y}^{6}-63090\,{x}^{12}{y}^{8}\\
&&+92252\,{x}^{10}{y}^{10}-
63090\,{x}^{8}{y}^{12}+19320\,{x}^{6}{y}^{14}-2445\,{x}^{4}{y}^{16}\\
&&+90\,{x}^{2}{y}^{18}-{y}^{20}.
\end{eqnarray*}
\end{itemize}
The timings corresponding to these curves are given in Table \ref{tab:example-curves}.

\begin{table}
\begin{center}
\begin{tabular*}{\columnwidth}{l@{\extracolsep{\stretch{1}}}*{6}{c}@{}}
\toprule
Curve & Descartes'& Bernoulli's & Epitrochoid & Cardioid & Hypocycloid\\
& folium & \!lemniscate\!& & offset & \\
 &
\includegraphics[scale=0.7]{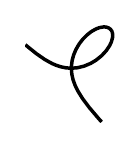}  & \includegraphics[scale=0.7]{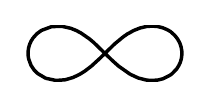} &  \includegraphics[scale=0.7]{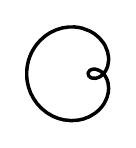} & \includegraphics[scale=0.7]{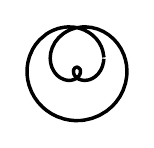} & \includegraphics[scale=0.7]{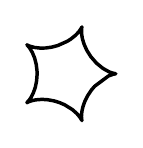} \\
degree   &          3 &                 4 &           4 &               8 &    8 \\
CPU time &       0.46 &              0.39 &        0.46 &       2.62       & 2.37 \\
\midrule
Curve & 4-leaf rose & 8-leaf rose & 12-leaf rose & 16-leaf rose & 20-leaf rose \\
 &
\includegraphics[scale=0.7]{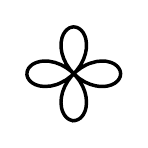} & \includegraphics[scale=0.7]{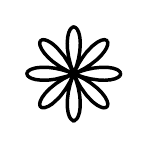} &  \includegraphics[scale=0.7]{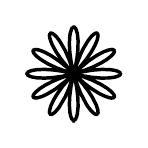}  & \includegraphics[scale=0.7]{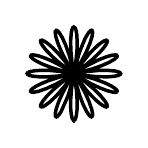} & \includegraphics[scale=0.7]{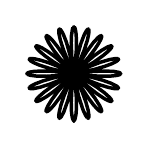} \\ 
degree      & 6 & 10 & 14 & 18 & 22 \\ 
CPU time    & 0.75 & 0.79 &1.56 & 5.67 & 9.85 \\ 
\bottomrule
\end{tabular*}
\end{center}
\caption{Average CPU time (seconds) of the algorithms for well-known curves.}\label{tab:example-curves}
\end{table}

\end{document}